\documentclass[11pt, english, draft]{article}
\usepackage{amsmath,graphicx,amssymb,amsfonts,tikz}
\usepackage{graphicx,color}
\usetikzlibrary{shapes.geometric}

\newtheorem{prethm}{{\bf Theorem}}

\newenvironment{thm}{\begin{prethm}{\hspace{-0.5
               em}{\bf}}}{\end{prethm}}
\newtheorem{prepro}{{\bf Theorem}}

\newtheorem{precon}{{\bf Conclusion }}
\newenvironment{con}{\begin{precon}{\hspace{-0.5
               em}{\bf}}}{\end{precon}}
\newtheorem{preprop}{{\bf Proposition}}

\newenvironment{prop}{\begin{preprop}{\hspace{-0.5
               em}{\bf}}}{\end{preprop}}
\newtheorem{preex}{{\bf Example}}

\newtheorem{precor}{{\bf Corollary}}

\newtheorem{preconj}{{\bf Conjecture}}

\newtheorem{preremark}{{\bf Remark}}

\newtheorem{prefact}{{\bf Fact}}

\newtheorem{prelem}{{\bf Lemma}}

\newenvironment{lem}{\begin{prelem}{\hspace{-0.5
               em}{\bf}}}{\end{prelem}}
\newtheorem{preproof}{{\em Proof.}}

\newenvironment{proof}[1]{\begin{preproof}{\rm
               #1}\hfill{$\Box$}}{\end{preproof}}

\newcommand{\diam}{\textrm{diam}}
\newcommand{\NEPS}{\textrm{NEPS}}
\newcommand{\Spec}{\textrm{Spec}}

\title{\bf Some non-sign-symmetric signed graphs with symmetric spectrum}
\author{\bf  F. Ramezani \and
  \\ {\it Faculty of Mathematics,
 K. N. Toosi University of Technology,}\\
 {\it  P. O. Box $16315$--$1618$,  Tehran, Iran}
 \\[.3cm]
 {\it E-mail:} \ {\tt ramezani@kntu.ac.ir}
 \date{}}

\begin{document}
\maketitle

\begin{abstract}
We construct infinitely many signed graphs having symmetric spectrum, by using the NEPS and rooted product of signed graphs. We also present a method for constructing large cospectral signed graphs. Although the obtained family contains only a minority of signed graphs, it strengthen the belief that the signed graphs with symmetric spectrum are deeper than bipartite graphs, i.e the unsigned graphs with symmetric spectrum.
\end{abstract}
 \vspace*{1cm}

\section{Introduction}
 In this paper, we consider only undirected  simple graphs (loops and multiple
edges are not allowed). Through out this article let $n$ be the order, and $\{v_1,v_2,\ldots,v_n\}$ be the vertex set of the graph $G$. By $v_i\sim v_j$ we mean these two vertices are adjacent in the graph, otherwise we use $v_i\nsim v_j$. By $d_{G}(u,v)$ we denote the distance between the vertices $u,v$ in the graph $G$. The notation $\diam(G)$ stands for the diameter of a graph $G$.

A graph $G=(V,E)$ provided with a sign function
$\sigma:E\longrightarrow \{+,-\}$ is called a
\textit{signed graph}. We denote a signed graph on $G$ with sign function $\sigma$,
by $\Sigma=(G,\sigma)$.
The graph $G$ is called the \textit{ground}
of $\Sigma$ and the edge set $E^{-}_{\Sigma}=\sigma^{-1}(-1)$ is called the
\textit{signature} of it. By the edge and vertex set of $\Sigma$ we mean
those of the ground graph that are $V,E$ respectively. For any edge $e$ of $\Sigma$, we call it a
\textit{positive} or \textit{negative} edge if $\sigma(e)$ has positive or negative sign
respectively. For $\Sigma=(G,\sigma)$, by $-\Sigma$ we mean the signed graph $(G,-\sigma)$, which is called \textit{negative} of $\Sigma$.
A signed graph with a bipartite ground will be called a \textit{signed bipartite graph}. A signed graph $(G,\sigma)$, all of whose edges have the same sign $+$ ($-$) is simply denoted by $G^+$ ($G^-$). By
\textit{resigning} at a vertex $v$ of $\Sigma$ we mean negating
signs of all the edges incident with $v$. Two signed graphs
$(G,\sigma)$ and $(G,\sigma')$ are called \textit{switching
equivalent} if one is obtained from the other by a sequence of
resignings. By $\Sigma\sim\Sigma'$, we mean that $\Sigma$ and $\Sigma'$ are switching equivalent.
For a function $f:A\rightarrow B$, and a subset $S$ of $A$, by $f|_S$ we mean the restriction of $f$ to $S$. For a signed graph $\Sigma=(G,\sigma)$ and a vertex $v$ of $G$,
by $\Sigma-v$ we mean the signed graph with ground $G-v$ where the corresponding sign function is $\sigma|_{E(G-v)}$. By a \textit{rooted graph} we mean a graph in which one specified vertex is distinguished as the root. A \textit{rooted signed graph} is a signed graph with a rooted ground.

With $A_{\Sigma}$ we denote the adjacency matrix of $\Sigma$, which is an $n\times n$ matrix with the following entries: $$A_{\Sigma}(i,j)=\left\{
                                                                                                         \begin{array}{ll}
                                                                                                           \sigma(v_iv_j), & \hbox{$v_i\sim v_j$;} \\
                                                                                                           0, & \hbox{$v_i\nsim v_j$.}
                                                                                                         \end{array}
                                                                                                       \right.
$$ We denote the characteristic polynomial of $\Sigma$, i.e $\det(\lambda I-A_{\Sigma})$, by $\chi(\Sigma)=\chi({\Sigma, \lambda})$.
The multiset of eigenvalues of $A_{\Sigma}$ is called the {\it spectrum} of $\Sigma$ and denoted by $\Spec(\Sigma)$.
Since $A_{\Sigma}$ is a symmetric matrix, the eigenvalues of $A_{\Sigma}$ (or ${\Sigma}$)
are real. We say that a signed graph $\Sigma$ has a \textit{symmetric spectrum} if for
any $\lambda\in \Spec(\Sigma)$,  $-\lambda$ belongs to $\Spec(\Sigma)$ with the same multiplicity. It is obvious that signed graphs with bipartite ground has a symmetric spectrum. For more details on the terminologies see \cite{ZAS}.

 We say that a function $\varphi:V(G)\rightarrow V(H)$ is an \textit{isomorphism} between signed graphs $(G,\sigma)$ and $(H,\pi)$ if the followings hold.
\begin{itemize}
  \item $\varphi$ is a graph isomorphism between $G$ and $H$,
  \item $\sigma(uv)=\pi(\varphi(u)\varphi(v))$.
\end{itemize}
Two signed graphs are called \textit{isomorphic}, denoted by $\Sigma\simeq\Pi$, if there is an isomorphism between them. Otherwise we call them \textit{non-isomorphic}. $\Sigma=(G,\sigma)$ and $\Pi=(H,\pi)$ are said to be \textit{switching isomorphic} if there exists an isomorphic image of $\Sigma$, say $\Pi'$, such that $\Pi'\sim \Pi$. Otherwise we call them \textit{switching non-isomorphic}. By $\Sigma\cong\Pi$, we mean that $\Sigma$ is switching isomorphic to $\Pi$.

A signed graph $\Sigma$ is called \textit{sign-symmetric} if $\Sigma\cong-\Sigma$. It is known that sign-symmetric graphs have symmetric spectrum.  The signed graphs with symmetric spectrum have attracted many attentions because of the folklore result on the ordinary graphs which states a graph admits a symmetric spectrum if and only if it is bipartite, see for instance \cite{BH} for ordinary, and \cite{FR} for signed graphs. For signed graphs it seems a much more richer family of them may admit symmetric spectrum. The problem of determining the properties of signed graphs with symmetric spectrum is proposed in \cite{BCKW,EF}. In \cite{EF}, the authors have presented the following example $(K_8,\sigma)$ of a non-sign-symmetric signed graph, having symmetric spectrum ($K_8$ is the complete graph on $8$ nods). We refer to the signed graph of Figure \ref{sk} by $SK_8$. Note that in $SK_8$, shown in Figure 1, red edges are negative and blue edges are positive edges of $SK_8$.

\vspace{-0.5cm}
\hspace{4.5cm}
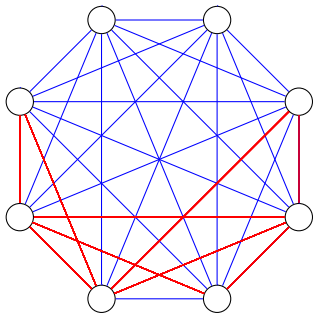
\begin{figure}[h] \label{sk}
 { \centering

\begin{tikzpicture}[mystyle/.style={draw,shape=circle}]
\def\ngon{8}
\node[regular polygon,regular polygon sides=\ngon,minimum size=4cm] (p) {};
\foreach\x in {1,...,\ngon}{\node[mystyle] (p\x) at (p.corner \x){};}
\foreach\x in {1,...,\numexpr\ngon-1\relax}{
\foreach\y in {\x,...,\ngon}{
\draw (p\x) -- (p\y)[blue];
\draw (p3) -- (p4)[red];
\draw (p3) -- (p5)[red];
\draw (p4) -- (p5)[red];
\draw (p4) -- (p6)[red];
\draw (p4) -- (p7)[red];
\draw (p5) -- (p7)[red];
\draw (p5) -- (p8)[red];
\draw (p6) -- (p7)[red];
\draw (p7) -- (p8)[red];}
\vspace{3cm}
}
\end{tikzpicture}
\caption {$SK_8:$\textrm{ A non-sign-symmetric signed graph with symmetric spectrum}}}
\end{figure}

In fact signed graphs with symmetric spectrum which are not sign-symmetric is of main interest, also it is interesting to determine all the signed graphs with symmetric spectrum.
%

\section{Preliminaries}

We recall two main definitions of signed graph products, namely non-complete extended p-sum ($\NEPS$) and rooted product of signed graphs.
\subsection{NEPS of signed graphs}
The \textit{non-complete extended p-sum} or \textit{NEPS} of signed graphs is defined in \cite{GERZAS}. It is based on the Cvetkovi$\acute{\textrm{c}}$ product for ordinary graphs, see \cite{CVE}. The following is their definition. This product is defined in terms of a basis for the product, that is $\mathcal{B}\subseteq \{0,1\}^k$. Set $\Sigma_i=(\Gamma_i,\sigma_i)$, for $i=1,\ldots,k$. First it is defined for one arbitrary vector
$\beta = (\beta_1, \beta_2,\ldots,\beta_k ) \in \{0, 1\}^k$. This product, written $\NEPS(\Sigma_1,\ldots,\Sigma_k ; \beta)$, is the signed
graph $(\Gamma,\sigma)$ with vertex set: $V_{\Gamma} = V_1 \times V_2 \times\cdots\times V_k$, and edge set:
$$E_{\Gamma} = \{(u_1,\ldots, u_k )(v_1,\ldots, v_k ) \textrm{ : } u_i = v_i \textrm{ if } \beta_i = 0 \textrm{ and } u_iv_i \in E_i \textrm{ if } \beta_i = 1\},$$
and the sign function is as following: $$\sigma((u_1,\ldots, u_k )(v_1,\ldots, v_k ))=\prod_{i=1}^{k}\sigma_i(u_iv_i)^{\beta_i}=\prod_{i:\beta_i=1}\sigma_i(u_iv_i).$$
 In the general definition there is a set $\mathcal{B} = \{\beta_1,...,\beta_q\}\subseteq\{0, 1\}^k \setminus \{(0, 0,..., 0)\}$ and it is defined:
$$\NEPS(\Sigma_1,..., \Sigma_k ; \mathcal{B}) = \bigcup_{\beta\in \mathcal{B}} \NEPS(\Sigma_1,\ldots,\Sigma_k ; \beta),$$
the underlying graph of which is the Cvetkovi$\acute{\textrm{c}}$ product $\NEPS(\Gamma_1, \Gamma_2,..., \Gamma_k ; \mathcal{B})$ of
the underlying graphs as defined in \cite[Section 2.5]{CVE}. At \cite{GERZAS} the authors have also evaluated the eigenvalues of $\NEPS(\Sigma_1, \Sigma_2,..., \Sigma_k ; \mathcal{B})$ based on eigenvalues of the factors. Their result follows.
\begin{thm} \label{zas1} {\rm {\bf\cite{GERZAS}}
  The eigenvalues of $\NEPS(\Sigma_1, \Sigma_2,..., \Sigma_k ; \mathcal{B})$ are $$\lambda_{j_1\ldots j_k}=\sum_{\beta\in \mathcal{B}}\lambda_{1j_1}^{\beta_1}\ldots \lambda_{kj_k}^{\beta_k},$$ for $1\leq j_1\leq  n_1,\ldots, 1\leq j_k \leq n_k$, where $n_i$ is the vertex number of $\Sigma_i$ and $\lambda_{ij_i}$ is the $j_i$th eigenvalue of $\Sigma_i$, for $i=1,2,\ldots,k$.  }
\end{thm}

%
%
\subsection{Rooted product}
The notion of rooted product of graphs which is defined by Godsil and Mckay in
\cite{GM}, plays an essential role for constructing signed graphs with symmetric spectrum. Let $\mathcal{H}$ be a sequence of rooted graphs $H_{1}, H_{2},\ldots,
H_{n}$ with the corresponding roots $r_1,r_2,\ldots,r_n$,
respectively. Recall $V=\{ v_{1}, v_{2},\ldots, v_{n}\}$, the vertex set  of  graph $G$. The {\it
rooted product of $G$ by $\mathcal{H}$},
denoted by $G[\mathcal{H}]$, is obtained from $G$ by
identifying each vertex $v_{i}$ by the root of $H_{i}$. If all the rooted graphs $H_i$ are isomorphic to a rooted graph $H$ then we simply write $G[H]$ instead of $G[\mathcal{H}]$. The matrix $A_{\lambda}G[\mathcal{H}]$ is defined as follows: $$A_{\lambda}(G[\mathcal{H}])=(a_{i,j}),$$ $$a_{i,j}=\left\{
            \begin{array}{ll}
             \chi(H_i,\lambda), & \hbox{$i=j$;} \\
              -A_G(i,j)\chi(H_i-r_i,\lambda), & \hbox{otherwise.}
            \end{array}
          \right.
$$
The characteristic polynomial of $G[\mathcal{H}]$
is given  by Godsil and McKay. The following is their result.
\begin{thm} \label{goma}  {\bf \cite{GM}}
$\chi(G[\mathcal{H}],\lambda)={\rm det } A_{\lambda}(G[\mathcal{H}]).$
\end{thm}
The special case where all the rooted graphs $H_i$ are isomorphic, is considered in \cite{sw}. We state their result in the following.
\begin{thm} \label{suw}  {\bf \cite{sw}} {\rm Let $H$ be a rooted graph with root $r$, then:}
$$\chi(G[H],\lambda)={\chi(H-r,\lambda)}^{n}\chi(G,\frac{\chi(H,\lambda)}{\chi(H-r,\lambda)}).$$
\end{thm}

\subsubsection{Rooted product of signed graphs}
We first generalize the notion of rooted product to signed graphs. Let $\Sigma=(G,\sigma)$ be a signed graph on $n$ vertices, and $\tilde{\mathcal{H}}$ be a list of $n$ rooted signed graphs say $\Pi_i=(H_i,\sigma_i)$, $i=1,\ldots,n$. Their corresponding rooted product, that is $\Sigma[\tilde{\mathcal{H}}]$, is defined to be the signed graph obtained by identifying the vertex $v_{i}$ of $\Sigma$ by the root of $\Pi_i$, for $i=1,2,\ldots,n$. If all signed graphs of the list $\tilde{\mathcal{H}}$ are isomorphic to the signed graph $\Pi$, then we simply denote the rooted product of $\Sigma$ and $\tilde{\mathcal{H}}$ by $\Sigma[\Pi]$.
We state the following proposition without proof. It is a similar result as Theorem \ref{goma} where the graphs are replaced with signed graphs. Note that it can be verified by a similar method of Godsil-McKay in their proof of Theorem \ref{goma}, see \cite{GM}. The difference is that we must pay attention to the edge signs. In the following let $A_{\lambda}(\Sigma[\tilde{\mathcal{H}}])$ be the bellow matrix: $$A_{\lambda}(\Sigma[\tilde{\mathcal{H}}])=(a_{i,j}),$$ $$a_{i,j}=\left\{
            \begin{array}{ll}
             \chi(\Pi_i,\lambda), & \hbox{$i=j$;} \\
              -A_{\Sigma}(i,j)\chi(\Pi_i-r_i,\lambda), & \hbox{otherwise.}
            \end{array}
          \right.
$$
\begin{prop} \label{sgoma}
$\chi(\Sigma[\tilde{\mathcal{H}}],\lambda)={\rm det } A_{\lambda}(\Sigma[\tilde{\mathcal{H}}]).$
\end{prop}

%
\section{Main Results}
Our main result is consisting of two independent concepts. At first we present methods for constructing signed graphs with symmetric spectrum. Afterwards we present a method for constructing switching non-isomorphic cospectral signed graphs. There is an obvious method of constructing non-sign-symmetric graphs with symmetric spectrum which yields disconnected signed graphs. The method follows. Let $\Gamma_1$, $\Gamma_2$ be two non-isomorphic, non-bipartite, cospectral graphs, then the signed graph $\Sigma=\Gamma_1^+\cup \Gamma_2^-$ (disjoint union of $\Gamma_1^+$ and ${\Gamma_2}^-$) has symmetric spectrum but it is not sign-symmetric. We are interested in connected examples.

\subsection{Signed graphs with symmetric spectrum: NEPS}
The notion of NEPS of signed graphs can be applied for constructing signed graphs with symmetric spectrum. In \cite{CVE1}, and some of its references an analogous problem for ordinary graphs is considered. Let $\tilde{\Sigma}:=\Sigma_1, \Sigma_2,..., \Sigma_k$ be a list of signed graphs. We say that $\tilde{\Sigma}$ is \textit{symmetric} if for any $i=1,\ldots,k$, there is a unique $j$ for which $\Spec(\Sigma_i)=-\Spec(\Sigma_j)$, in which case we write $j=i^-$. For a symmetric list of signed graphs $\tilde{\Sigma}$, we say that a base $\mathcal{B}$ is \textit{compatible} with $\tilde{\Sigma}$ if for each $\beta\in\mathcal{B}$ there is a unique $\beta'\in\mathcal{B}$ so that the following equality holds: $$\{j: \beta'_j=1\}=\{i^-: \beta_i=1\}.$$
In the following theorem we present a sufficient condition for NEPS of signed graphs to admit a symmetric spectrum.
\begin{thm} {\rm
 Let $\Sigma_1,\Sigma_2,\ldots,\Sigma_k$ be given signed graphs. For $\mathcal{B}\subseteq \{0,1\}^k$, the signed graph $\Theta=\NEPS(\Sigma_1, \Sigma_2,..., \Sigma_k ; \mathcal{B})$ admits a symmetric spectrum if both of the followings are satisfied.

\begin{itemize}
\item The list $\Sigma_1,\Sigma_2,\ldots,\Sigma_k$ is symmetric and the base $\mathcal{B}$ is compatible with it.
\item The function $p$ defined by $p(x_1,\ldots,x_k)=\sum_{\beta\in \mathcal{B}}x_{1}^{\beta_1}\ldots x_{k}^{\beta_k}$ is odd, that is,  $$p(-x_1,\ldots,-x_k)=-p(x_1,\ldots,x_k).$$
 \end{itemize}
 }
\end{thm}
\begin{proof} {Note that by Theorem \ref{zas1}, any eigenvalue $\lambda$ of $\Theta$ is equal to $p(\lambda_{1j_1},\ldots,\lambda_{kj_k})$ for some $1\leq j_1\leq  n_1,\ldots, 1\leq j_k \leq n_k$. The function $p$ is an odd function, so $$p(-\lambda_{1j_1},\ldots,-\lambda_{kj_k})=-p(\lambda_{1j_1},\ldots,\lambda_{kj_k})=-\lambda.$$ Now we prove that $p(-\lambda_{1j_1},\ldots,-\lambda_{kj_k})$ is an eigenvalue of $\Theta$. Note that $\Spec(\Sigma_i)=-\Spec(\Sigma_{i^-})$, so $-\lambda_{ij_i}=\lambda_{i^-j'_{i^-}}$, for some $1\leq j'_{i^-}\leq n_i$. This implies the following equality: $$p(-\lambda_{1j_1},\ldots,-\lambda_{kj_k})= p(\lambda_{1^-j'_{1^-}},\ldots,\lambda_{k^-j'_{k^-}}).$$

By the definition we have the following: $$p(\lambda_{1^-j'_{1^-}},\ldots,\lambda_{k^-j'_{k^-}})=\sum_{\beta\in \mathcal{B}}\lambda_{1^-j'_{1^-}}^{\beta_1}\ldots \lambda_{k^-j'_{k^-}}^{\beta_k}.$$

Now consider the summand $\lambda_{1^-j'_{1^-}}^{\beta_1}\ldots \lambda_{k^-j'_{k^-}}^{\beta_k}.$ It is equal to $\prod_{i:\beta_i=1}\lambda_{i^-j'_{i^-}}$. On the other hand by definition we have $\{i: \beta'_i=1\}=\{i^-: \beta_i=1\}.$  This yields, $$\prod_{i:\beta_i=1}\lambda_{i^-j'_{i^-}}=\prod_{i:\beta'_i=1}\lambda_{ij'_{i}},$$
which finally turns out that $$\sum_{\beta\in \mathcal{B}}\lambda_{1^-j'_{1^-}}^{\beta_1}\ldots \lambda_{k^-j'_{k^-}}^{\beta_k}=\sum_{\beta'\in \mathcal{B}}\lambda_{1j'_{1}}^{\beta'_1}\ldots \lambda_{kj'_{k}}^{\beta'_k}.$$
 But $\beta$'s and $\beta'$'s are in a one to one correspondence, thus $$\sum_{\beta'\in \mathcal{B}}\lambda_{1j'_{1}}^{\beta'_1}\ldots \lambda_{kj'_{k}}^{\beta'_k}=p(\lambda_{1j'_{1}},\ldots, \lambda_{kj'_{k}}),$$ which is an eigenvalue of $\Theta$ by Theorem \ref{zas1}. Therefore $-\lambda$ is an eigenvalue of $\Theta$, as desired.}
\end{proof}
We are now ready to introduce some examples of signed graphs with symmetric spectrum based on NEPS of signed graphs.

\textbf{Example}: Let $\Gamma_1,\Gamma_2$ be two non-isomorphic cospectral graphs (graphs with the same adjacency spectrum). Then by the above theorem the graph $\NEPS(\Gamma_1^+,\Gamma_2^-;\{(0,1),(1,0)\})$, which is in fact the Cartesian product $\Gamma_1^+\times\Gamma_2^-$ of the mentioned signed graphs, admits a symmetric spectrum.
We may choose $\Gamma_1,\Gamma_2$ so that the resulting signed graph becomes non-sign symmetric. We don't go through details, but we refer the reader to apply methods in \cite{IMW}.


\subsection{Signed graphs with symmetric spectrum: Rooted product}
We have the following result on the rooted product of signed graphs.
\begin{prop} {\rm Let $\Pi$ be a rooted signed graph with root $r$ and $\Sigma$ be a signed graph.
If all the signed graphs $\Sigma, \Pi, \Pi-r$, have symmetric spectrum, then the signed graph $\Sigma[\Pi]$ has also symmetric spectrum.}
\end{prop}

\begin{proof}
{By Proposition \ref{sgoma}, the characteristic polynomial of the signed graph  $\Sigma[\Pi]$ is equal to the determinant of the matrix $A_{\lambda}(\Sigma[\Pi])$. Note each of the diagonal entries of $A_{\lambda}(\Sigma[\Pi])$ is $\chi(\Pi,\lambda)$. And entries elsewhere are equal to $-A_{\Sigma}(i,j)\chi(\Pi-r,\lambda)$, by the definition. Since $\Pi$ has symmetric spectrum, one of the following two cases may occur for $\Pi$.
\begin{itemize}
  \item $\Pi$ has an even number of vertices and $\chi(\Pi,\lambda)=\sum_{i=0}^{\frac{n}{2}}{{a_{2i}\lambda}^{2i}}$.
  \item It has an odd number of vertices and $\chi(\Pi,\lambda)=\lambda\sum_{i=0}^{\frac{n-1}{2}}{{b_{2i}\lambda}^{2i}}$.
\end{itemize}
We continue the proof concerning the first case, the second one can be proved with a similar approach. By the assumption the matrix $A_{\lambda}(\Sigma[\Pi])$ is equal to the following matrix. $$\chi(\Pi,\lambda)I_n-\chi(\Pi-r,\lambda)A_{\Sigma},$$
where $I_n$ is the $n\times n$ identity matrix. But we supposed that $\Pi$ has even number of vertices, so $\Pi-r$ has an odd number of vertices and we may assume the following equalities. $$\chi(\Pi,\lambda)I_n-\chi(\Pi-r,\lambda)A_{\Sigma}=(\sum_{i=0}^{\frac{n}{2}}{{a_{2i}\lambda}^{2i}})I_n-(\lambda\sum_{i=0}^{\frac{n-2}{2}}{{b_{2i}\lambda}^{2i}})A_{\Sigma}$$

By dividing each of the rows of the above matrix by the polynomial $\lambda\sum_{i=0}^{\frac{n-2}{2}}{{b_{2i}\lambda}^{2i}}$, determinant of the matrix $\chi(\Pi,\lambda)I_n-\chi(\Pi-r,\lambda)A_{\Sigma}$ will be equal to the following expression. $$P(\lambda)=(\lambda\sum_{i=0}^{\frac{n-2}{2}}{{b_{2i}\lambda}^{2i}})^n\det(\frac{\sum_{i=0}^{\frac{n}{2}}{{a_{2i}\lambda}^{2i}}}{\lambda\sum_{i=0}^{\frac{n-2}{2}}{{b_{2i}\lambda}^{2i}}}I_n-A_{\Sigma})$$  Now it suffices to prove that for a real number $\lambda_0$, $P(\lambda_0)=0$ if and only if $P(-\lambda_0)=0$. Suppose that $P(\lambda_0)=0$ for a real number $\lambda_0$ then at least one of the followings holds. $$\lambda_0\sum_{i=0}^{\frac{n-2}{2}}{{b_{2i}\lambda_0}^{2i}}=0, \textrm{ or }\hspace{1cm}(1)$$ $$\det(\frac{\sum_{i=0}^{\frac{n}{2}}{{a_{2i}\lambda_0}^{2i}}}{\lambda_0\sum_{i=0}^{\frac{n-2}{2}}{{b_{2i}\lambda_0}^{2i}}}I_n-A_{\Sigma})=0\hspace{1cm}(2)$$
If (1) holds then the assertion follows easily, so suppose (2) holds. Note that the equality (2) holds if and only if for some eigenvalue $\mu_0$ of $A_{\Sigma}$ the following equality carries:
$$\frac{\sum_{i=0}^{\frac{n}{2}}{{a_{2i}\lambda_0}^{2i}}}{\lambda_0\sum_{i=0}^{\frac{n-2}{2}}{{b_{2i}\lambda_0}^{2i}}}=\mu_0.$$
Hence we have the following: $$\frac{\sum_{i=0}^{\frac{n}{2}}{{a_{2i}(-\lambda_0)}^{2i}}}{-\lambda_0\sum_{i=0}^{\frac{n-2}{2}}{{b_{2i}(-\lambda_0)}^{2i}}}=-\mu_0.$$
On the other hand $A_{\Sigma}$ has symmetric spectrum. Thus $-\mu_0$ is an eigenvalue of $A_{\Sigma}$, hence $P(-\lambda_0)=0$. This implies that the signed graph $\Sigma[\Pi]$ admits a symmetric spectrum.  }
\end{proof}

\begin{lem}  {\rm Let $\Pi$ be a signed bipartite graph with root $r$. The signed graph $SK_8[\Pi]$ is not sign-symmetric.}
\end{lem}
\begin{proof}
 { Note that both the signed graphs $SK_8[\Pi]$ and $-SK_8[\Pi]$ have a unique $8$-clique in their grounds, i.e $K_8$. Let $\varphi$ be an isomorphism between $SK_8[\Pi]$ and $-SK_8[\Pi]$, then the image of $SK_8$ under $\varphi$ must be $-SK_8$. This means that $SK_8$ and $-SK_8$ are isomorphic, which is impossible. Hence the two signed graphs $SK_8[\Pi]$ and $-SK_8[\Pi]$ are not isomorphic.   }
\end{proof}

\begin{con} {\rm The signed graph $SK_8[\Pi]$ for a signed bipartite graph $\Pi$ or even for a signed graph with symmetric spectrum, provides an infinite list of signed graphs with symmetric spectrum. Note that the center can be replaced with any non-sign-symmetric signed graph with symmetric spectrum. Hence the method can be applied to construct infinitely many examples of signed graphs with the mentioned property.}
\end{con}
\subsection{Constructing cospectral signed graphs}
The notion of rooted product of graphs can be applied to construct cospectral graphs as well, see \cite{GRT,sw} for example. We extend the method for signed graphs. Accordingly, two signed graphs $\Sigma_1$ and $\Sigma_2$ are called \textit{cospectral} if they share the same adjacency spectrum. Two rooted graphs $G,H$ are called \textit{cospectrally rooted} if they are cospectral and the subgraphs $G-u$ and $H-v$ are also cospectral, where $u,v$ are the corresponding roots.

\begin{prop} \label{cos}
  {\rm Let $\Sigma$ be an arbitrary signed graph and $G,H$ be two non-isomorphic cospectrally rooted graphs with roots $u,v$ respecttively. Then the signed graphs $\Sigma[G^+]$ and $\Sigma[H^+]$ are cospectral.}
\end{prop}
\begin{proof}{
The characteristic polynomial of $\Sigma[G^+]$ is equal to determinant of the matrix $A_{\lambda}(\Sigma[G^+])$ by Proposition \ref{sgoma}. The entries of $A_{\lambda}(\Sigma[G^+])$ are the followings:  $$a_{i,j}=\left\{
            \begin{array}{ll}
             \chi(G^+,\lambda), & \hbox{$i=j$;} \\
              -A_{\Sigma}(i,j)\chi(G^+-u,\lambda), & \hbox{otherwise.}
            \end{array}
          \right.
$$
Note that $G^+,H^+$  are cospectrally rooted therefore the entries of the two matrices $A_{\lambda}(\Sigma[G^+])$ and $A_{\lambda}(\Sigma[H^+])$ coincide. Hence they have equal determinants which implies the assertion.}
\end{proof}
 As an example of cospectrally rooted graphs, see the following figure.
\begin{figure}[h]\vspace{-3cm}\hspace{.5cm}
\unitlength 1.20mm \linethickness{0.4pt}
\begin{picture}(80,60)(-25,80)
 \put (0,110) {\line (1,0) {10}}\put(10,110) {\line (1,-2) {4}}
\put(10,94) {\line(1,2){4}} \put (10,94 ) {\line (-1,0) {10}} \put
(0 ,94) {\line (-1,2) {4}}\put (0,110 ) {\line (-1,-2) {4}} \put
(60,112) {\line (0,-1) {12}}\put(60,100) {\line (2,-1) {11}}
\put(60,100) {\line(-2,-1){11}} \put (10,94 ) {\line (-1,0) {10}}
\put (0 ,94) {\line (-1,2) {4}}\put (0,110 ) {\line (-1,-2) {4}}

\put(0,110
  ){\circle*{1.5}}\put(10,110){\circle*{1.5}}\put(0,94
  ){\circle*{1.5}}\put(10,94){\circle*{1.5}}
  \put(14,102){\circle*{1.5}}\put(-4,102){\circle*{1.5}}\put(5,102){\circle*{1.5}}
\put(60,112
  ){\circle*{1.5}}\put(60,106){\circle*{1.5}}\put(60,100
  ){\circle*{1.5}}\put(65.5,97.5){\circle*{1.5}}
  \put(54.5,97.5){\circle*{1.5}}\put(71,94.5){\circle*{1.5}}\put(49,94.5){\circle*{1.5}}
\put(12,111){\makebox(0,0)[cc]{$u$}}
\put(5,90){\makebox(0,0)[cc]{$G$}}
\put(58,107){\makebox(0,0)[cc]{$v$}}
\put(60,91){\makebox(0,0)[cc]{$H$}}
\end{picture}
\vspace{-1cm}
\caption{Cospectrally rooted graphs}\label{2}
\end{figure}
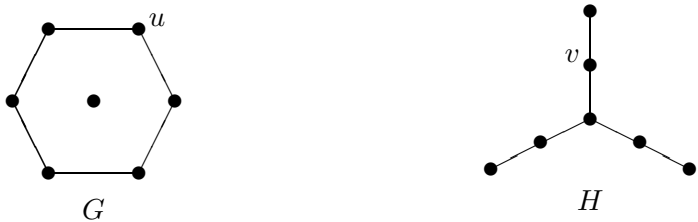

{\bf Ending Remark. } In Proposition \ref{cos} we considered the root product of $\Sigma$ by a list consisting of a unique graph. But this applies in a wider manner. Let $\tilde{\mathcal{H}_1}$ and $\tilde{\mathcal{H}_2}$  be the lists ${H_1}^+,\ldots,{H_n}^+$ and ${H'_1}^+,\ldots,{H'_n}^+$ respectively. Suppose that for any $i=1,\ldots,n$ the graphs $H_i$ and $H'_i$ are either cospectrally rooted or isomorphic (we call this the \textit{coiso condition}) then the same result as mentioned in Proposition \ref{cos} holds true. In fact the signed graphs $\Sigma[\tilde{\mathcal{H}_1}]$ and $\Sigma[\tilde{\mathcal{H}_2}]$ remains cospectral. This follows by a similar approach which is used in the proof of Proposition \ref{cos}. The results also remains valid if we replace the all positive signed graphs with all negative ones, or even with arbitrary signed graphs providing the coiso condition. Note that for providing non isomorphic signed graphs as result, not all the graphs in the lists should be isomorphic.


\end{document}